\documentclass[reqno]{article}

\usepackage[top=2.5cm, bottom=2.5cm, left=3cm, right=3cm]{geometry}

\usepackage[english]{babel}
\usepackage[utf8]{inputenc}

\usepackage{upgreek}

\usepackage{marginnote}
\usepackage{lmodern}
\usepackage{amssymb}
\usepackage{amsmath}
\usepackage{mathtools}
\usepackage{amsthm}
\usepackage[usenames,dvipsnames]{xcolor}
\usepackage{graphicx}
\usepackage{tikz}
\usepackage{pgf}
\usepackage{subcaption}
\usepackage{wrapfig}
\usepackage{array}
\usepackage{tabularx}
\usepackage{titlesec}
\usepackage{eepic}
\usepackage{multicol}
\usepackage[linktocpage=true]{hyperref}
\hypersetup{
    colorlinks=true,
    linkcolor=blue,
    filecolor=magenta,      
    urlcolor=cyan,
    pdftitle={Overleaf Example},
    pdfpagemode=FullScreen,
    }
\usepackage[hypcap=true]{caption}
\usepackage[cmtip,all]{xy}
\usepackage{enumitem}
\usepackage{leftidx}
\usepackage[mathscr]{euscript}
\usepackage{setspace}
\usepackage{needspace}

\usepackage{parskip}

\makeatletter
\def\thm@space@setup{%
  \thm@preskip=\parskip \thm@postskip=0pt
}
\makeatother

\allowdisplaybreaks

\usepackage{titlesec}
\titleformat{\section}[block]{\color{black}\large\bfseries\filcenter}{\thesection.}{0.5em}{}
\titleformat{\subsection}[hang]{\bfseries}{}{0.5em}{}

\numberwithin{equation}{section}

\newtheorem{theorem}{Theorem}[]
\newtheorem{lemma}[theorem]{Lemma}
\newtheorem{proposition}[theorem]{Proposition}
\newtheorem{corollary}[theorem]{Corollary}
\newtheorem{definition}[theorem]{Definition}

\newtheorem{remark}[theorem]{Remark}

\newtheorem{problem}{Problem}[]

\newtheorem{ltheorem}{Theorem}

\newtheorem{ldefinition}{Theorem}

\newtheorem{lcorollary}[ltheorem]{Corollary}

\newtheorem{llemma}[ltheorem]{Lemma}

\titleformat{\subsection}[runin]{\bfseries}{}{}{}[.]
\titleformat{\subsubsection}[runin]{\bfseries}{}{}{}[.]

\newcounter{introequation}

\makeatletter
\renewenvironment{proof}[1][\proofname]{%
   \par\pushQED{\qed}\normalfont%
   \topsep6\p@\@plus6\p@\relax
   \trivlist\item[\hskip\labelsep\bfseries#1\@addpunct{.}]%
   \ignorespaces
}{%
   \popQED\endtrivlist\@endpefalse
}
\makeatother

\newcommand{\equivalent}{\Longleftrightarrow}

\def\XXint#1#2#3{{\setbox0=\hbox{$#1{#2#3}{\int}$ }
\vcenter{\hbox{$#2#3$ }}\kern-.6\wd0}}

\def\1{\mathbf{1}}

\def\M{\mathcal{M}}

\def\N{\mathcal{N}}

\makeatletter
\newsavebox{\@brx}
\newcommand{\llangle}[1][]{\savebox{\@brx}{\(\m@th{#1\langle}\)}%
  \mathopen{\copy\@brx\mkern2mu\kern-0.9\wd\@brx\usebox{\@brx}}}
\newcommand{\rrangle}[1][]{\savebox{\@brx}{\(\m@th{#1\rangle}\)}%
  \mathclose{\copy\@brx\mkern2mu\kern-0.9\wd\@brx\usebox{\@brx}}}
\makeatother

\newcommand{\vertiii}[1]{
 {\left\vert\kern-0.25ex\left\vert\kern-0.25ex\left\vert #1 
  \right\vert\kern-0.25ex\right\vert\kern-0.25ex\right\vert}
}

\def\B{\mathcal{B}}

\newcommand{\EE}{\mathbf{E}}

\newcommand{\RR}{\mathbf{R}}

\newcommand{\ZZ}{\mathbf{Z}}

\def\H{\mathcal{H}}

\newcommand{\Ee}{\mathscr{E}}

\newcommand{\Pp}{\mathscr{P}}

\title{
 Maximal weak Orlicz types\\
 and the strong maximal on von Neumann algebras
}

\author{
  Adrián M. González-Pérez\thanks{
    Partially supported by Ramón y Cajal fellowship RYC2022-037045-I 
    (Ministerio de Ciencia, Spain).
  }, \\
  Javier Parcet
  \ and 
  Jorge Pérez García\thanks{
    Partially supported by pre-doctoral scholarship PRE2020-093245.
    The three authors were supported by the Severo Ochoa Grant CEX2023-001347-5(MICIU),
    PIE2023-50E106(CSIC) and PID2022-141354NB-I00(MICINN)
  }, \
}
\date{}

\begin{document}

\maketitle

\begin{abstract}
      Let $\EE_n: \M \to \M_n$ and $\EE_m: \N \to \N_m$ be two sequences of conditional expectations on finite von Neumann algebras. The optimal weak Orlicz type of the associated strong maximal operator $\Ee = (\EE_n\otimes \EE_m)_{n,m}$ is not yet known. In a recent work of Jose Conde and the two first-named authors
    , it was show  that $\Ee$ has weak type $(\Phi, \Phi)$ for a family of functions including $\Phi(t) = t \, \log^{2+\varepsilon} t$, for every $\varepsilon > 0$. In this article, we prove that the weak Orlicz type of $\Ee$ cannot be lowered below $L \log^2 L$, meaning that if $\Ee$ is of weak type $(\Phi, \Phi)$, then $\Phi(s) \not\in o(s \, \log^2 s)$. Our proof is based on interpolation. Namely, we use recent techniques of Cadilhac/Ricard to formulate a Marcinkiewicz type theorem for maximal weak Orlicz types. Then, we show that a weak Orlicz type lower than $L \log^2 L$ would imply a $p$-operator constant for $\Ee$ smaller than the known optimum as $p \to 1^{+}$.
\end{abstract}

\section*{Introduction}

Maximal inequalities in the context of noncommutative von Neumann algebras have a long history that traces back to the independent works of Cuculescu \cite{Cuculescu1971} on martingale sequences and Lance \cite{Lance1976Ergodic}, and later Yeadon \cite{Yeadon1977, Yeadon1980II}, on ergodic theory. More concretely, Cuculescu proved that, if $\EE_n: \N \to \N_n$ is a sequence of unital conditional expectations on a semifinite von Neumann algebra $(\N, \tau)$, then, for every positive element $x \in L_1(\N)$ and scalar $\lambda \geq 0$, there exists a projection $e_\lambda \in \Pp(\N)$ such that 
\begin{enumerate}[label={\rm \textbf{(\roman*)}}, ref={\rm {(\roman*)}}]
    \item \label{itm:weaktype1} 
    $e_\lambda \, s_n(x) \, e_\lambda \leq \lambda \, e_\lambda$, 
    for every $n \geq 0$, and
    \item \label{itm:weaktype2} 
    $\displaystyle{ \tau(\1 - e_\lambda) \lesssim \frac{\| x \|_1}{\lambda}}$.
\end{enumerate}
It is trivial to verify that, when $\, \N=L_\infty(\Omega)$ is Abelian, these two conditions are equivalent to the weak type $(1,1)$ of the maximal Doob operator $f \mapsto \sup_n |\EE_n[f]|$, see \cite{Doob1953stochastic}.  Like in the classical setting, the above noncommutative weak type $(1,1)$ implies a generalization of almost everywhere convergence, introduced by Lance, called \emph{bilateral almost uniform} convergence. Indeed, a sequence $(x_n)_n$ of $\tau$-measurable operators, see \cite{Terp1981lp} for the precise definition, converge to $x$ bilaterally almost uniformly if, for every $\varepsilon > 0$, there is a projection $e \in \Pp(\N)$ such that $\tau(\1 - e) < \varepsilon$ and
\[
  \big\| e \, (x_n - x)\, e \big\|_\infty \to 0.
\]
Asymmetrical versions of the above definition can be defined by placing the projection $e$ only on the left or on the right. Those would be referred to as \emph{almost uniform convergence}. 
Thus, Cuculescu's inequality implies the bilateral almost uniform convergence of $\EE_n[x]$ for every $x \in L_1(\N)$.

While the tools for handling noncommutative weak types of maximal operators were set up early on, the strong types only appeared later. The main idea here is that the classical maximal estimates can be interpreted as mixed-norm bounds in $L_p[\ell_\infty]$, ie
\[
    \Big\| \sup_n |f_n| \Big\|_p = \big\| (f_n)_n \big\|_{L_p[\ell_\infty]}.
\]
The noncommutative analogues of Banach space-valued $L_p$-spaces appeared firstly in the work of Pisier \cite{Pi1998}, who defined a way of constructing $L_p(\N;E)$, where $E$ needs to have a finer structure than that given by its Banach space norm
---called an \emph{operator space structure}--- see \cite{Pi2003, EffRu2000Book}. Also, the algebra $\N$ in Pisier's definition needed to be hyperfinite. More general von Neumann algebras were considered by Junge in \cite{Junge2004Fubini}, where noncommutative operator space-valued $L_p$-spaces were defined for QWEP von Neumann algebras. In the particular case of $L_p(\N; \ell_\infty)$, a definition independent of the approximation properties of $\N$ was reached in the works of Junge and Xu \cite{Jun2002Doob, JunXu2003, JunXu2005Best}. Namely, the space $L_p(\N; \ell_\infty)$ is defined as the elements admitting a factorization of the form
\[
   L_p(\N; \ell_\infty) = 
   L_{2p}(\N) \, \big( \N \bar\otimes \ell_\infty \big) \, L_{2p}(\N)
\]
where the norm is taken to be 
\[
  \big\| (x_n)_n \big\|_{L_p[\ell_\infty]}
  \, := \, 
  \inf \Big\{ \| a \|_{2p} \, \| (u_n)_n \|_{\ell_\infty[\N]} \, \| b\|_{2p} : x_n = a \, u_n \, b \Big\},
\]
see \cite{JunXu2003} for the details. Once the spaces $L_p(\M;\ell_\infty)$ were defined, the strong type $(p,p)$ for the maximal operator associated to a family of maps $x \mapsto S_n(x)$ can be expressed just as the boundedness of $S = (S_n)_n: L_p(\N) \to L_p(\N;\ell_\infty)$. 
These mixed-norm spaces satisfy many natural properties. For instance, it was show in a remarkable article of Junge \cite{Jun2002Doob} that the maximal function associated to a family of conditional expectation has strong type $(p,p)$ for $1<p$. Nevertheless, their compatibility with Marcinkiewicz interpolation is very subtle. The noncommutative generalization of the Marcinkiewicz interpolation theorem was obtained by Junge and Xu \cite{JunXu2003} and their proof turned out to be remarkably involved, see \cite{Dirksen2015IntMax} as well. In fact, and contrary to the classical case, the operators $S_n$ are required to be positivity-preserving and the operator norms in $L_p$ explode like 
\[
  c(p) : = \big\| S = (S_n)_n: L_p(\N) \to L_p(\N;\ell_\infty) \big\|
  \sim \big(p'\big)^2 \;\;\; \text{ as } \; p \to 1^{+},
\]
while in the classical case the constant $c(p)$ grows merely like $p'$ as $p \to 1^{+}$, where $p'$ will represent the conjugate exponent of $p$ in the forthcoming discussion. This growth was later shown to be optimal \cite{JunXu2005Best, Hu2009} and not a shortcoming of Junge/Xu's proof. One of the underlying reasons for this technicality has to do with the fact that the projections $e_\lambda$ in \ref{itm:weaktype1}--\ref{itm:weaktype2} should behave as noncommutative generalization of the measurable sets $\{ | \sup_n S_n(x) | \leq \lambda \}$ and thus it would be expected that the integral of $e_\lambda^\perp = 1 - e_\lambda$, as a positive measurable operator, should have an $L_p$-norm comparable to the mixed $L_p[\ell_\infty]$-norm of $(S_n(x))_n$. Nevertheless this is not the case as explained in \cite{JunXu2005Best} or \cite[Remark 3.5]{CondeGonPar2019}.

Recently, noncommutative maximal inequalities have received renewed attention. For instance, maximal ergodic theorems for group actions of polynomial-growth groups \cite{HongLiaoWang2019}, Fourier multipliers \cite{HongWangWang2020} and amenable groups \cite{CadilhacWang2022} have been obtained. Part of our work will be built on a recent article of Cadilhac and Ricard that has clarified many of the technical steps behind the original proof of the noncommutative Marcinkiewicz interpolation for maximal inequalities \cite{CadilhacRicard2023Inter}.


\subsection*{The strong maximal function}
Let $\M \bar\otimes \N$ be the spacial tensor product of two finite von Neumann algebras $(\M, \tau_\M)$ and $(\N, \tau_\N)$ endowed with normal and faithful tracial states and let $(\M_n)_{n}$ and $(\N_n)_n$ be ascending sequence of (unital) von Neumann subalgebras $\M_n \subset \M_{n+1} \subset \cdots \subset \M$ and $\N_n \subset \N_{n+1} \subset \cdots \subset \N$. We would usually refer to these ascending chains of subalgebras as \emph{filtrations}. We will furthermore assume that the increasing union of the subalgebras is weak-$\ast$ dense inside $\M$ and $\N$ respectively.
We will denote by $\EE_n \otimes \EE_m : \M \bar\otimes \N \to \M_n \bar\otimes \N_m$ the tensor product of the conditional expectations onto $\M_n$ and $\N_m$ respectively. Observe that, if both $\N$ and $\M$ are isomorphic to $L_\infty([0,1])$ and the subalgebras are given by $\N_n = \M_n = L_\infty([0,1];\Sigma_n)$, where $\Sigma_n$ is taken to be the $\sigma$-algebra generated by the dyadic sets of length $2^{-n}$, then the maximal operator associated to the family of conditional expectations $(\EE_n \otimes \EE_m)_{n,m}$ is a dyadic model for the classical \emph{strong maximal operator}. This operator was studied early on. Indeed, Jessen, Marcinkiewicz and Zygmund \cite{Jessen1935} proved, using our previous notation, that
\[
  f \in L \, \log L([0,1]^2)
  \;\;\;\; \implies \;\;\;\;
  \big(\EE_n \otimes \EE_m \big)[f] \to f \;\; \text{almost everywhere}.
\]
This was shown using an argument that relies on the $\limsup$ of a family of operators and the fact that the one-dimensional maximal function maps $L \log L$ into $L_1$. They also proved that $L \log L$ is the largest Orlicz class for which there is convergence almost everywhere to the original data. The optimality was proved by showing that almost everywhere convergence for an Orlicz class $L_\Phi$ implies a \emph{restricted weak Orlicz type} estimate. The full weak Orlicz type $(\Phi, \Phi)$, where $\Phi(s) = t \, ( 1 + \log_+ s )$ was later proved using geometric arguments by Córdoba and Feffermann \cite{CorFeff1975}. An alternative proof of the weak Orlicz type $(\Phi, \Phi)$ using a Fubini-type argument to combine two different operators with weak type $(1,1)$ was given in \cite{Guzman1972ProductBases}. 

It is natural to ask if these results hold for general finite von Neumann algebras. Let $\Phi: [0,\infty) \to [0,\infty]$ be a locally finite, non-decreasing and convex function with $\Phi(0) = 0$ and $\lim_{t \to \infty} \Phi(t) = \infty$. Recall the following definition of (maximal) weak Orlicz type $(\Phi, \Phi)$.

\begin{definition}
    \label{def: maximal weak Orlicz type}
    Let $(\N, \tau_\N)$ and $(\M, \tau_\M)$ be semifinite von Neumann algebras. A sequence $(S_n)_{n \geq 0}$ of operators $S_n: L_1(\N) \cap \N \to L_0(\M)$ is said to be of (maximal) weak type $(\Phi,\Phi)$ if there exist a $C\geq 0$ such that for every $x \geq 0$ any $\lambda > 0$ there is a projection $e_\lambda \in \Pp(\M)$ satisfying:
    \begin{enumerate}[label={\rm \textbf{(\roman*)}}, ref={\rm {(\roman*)}}]
        \item \label{itm:MaxOrlicz1} $\displaystyle \big\| e_\lambda \, S_n(x) \, e_\lambda \big\|_{\infty} \leq \lambda$ 
        for all $n\geq 0$, and
        \item \label{itm:MaxOrlicz2}
        $\displaystyle \tau_\M(1-e_\lambda) \leq \tau_\N\circ \Phi\Big( C \, \frac{|x|}{\lambda}\Big)$.
    \end{enumerate}
\end{definition}

We will denote the optimal constant $C \geq 0$
in Definition $\ref{def: maximal weak Orlicz type}$ by $[S]_{(\Phi, \Phi)}$. Clearly, when $\N=L_\infty(\Omega)$ is Abelian, the definition recovers the classical one. The following problem, whose solution would generalize the results of \cite{Jessen1935, Guzman1972ProductBases, CorFeff1975}, was stated in \cite{CondeGonPar2019}.

\begin{problem}[{\rm from \cite[Conjecture B3]{CondeGonPar2019}}] \normalfont
\label{prp:ProblemA}
    Let $(\M,\tau_\M)$ and $(\N, \tau_\N)$ be two finite von Neumann algebras and $\EE_n \otimes \EE_m: \M \bar\otimes \N \to \N_n \bar\otimes \M_m$ be the tensor product of two conditional expectations associated to filtrations $(\M_n)_n$ and $(\N_n)_n$.
    \begin{enumerate}[label={\rm \textbf{(\roman*)}}, ref={\rm {(\roman*)}}]
        \item \label{itm:ProblemA.1}
        Is the optimal weak Orlicz type of the maximal operator associated to $(\EE_n \otimes \EE_m)_{n,m}$ given by $(\Phi, \Phi)$, where $\Phi(s) = s \, (1 + \log_+ s)^2$ \, ?
        \item \label{itm:ProblemA.2}
        Is the largest Orlicz class $L_\Phi(\M \bar\otimes \N)$ for which $(\EE_n \otimes \EE_m)(x)$ converges bilaterally almost uniformly for every $x \in L_\Phi$ as $n, m \to \infty$ given by $\Phi(s) = s \, (1 + \log_+ s)^2$ \, ?
    \end{enumerate}
\end{problem}

Although the problem above makes perfect sense in the case of semifinite von Neumann algebras, the crux of the problems lies in the local integrability, thus we will restrict our discussion to the finite case. Observe that, contrary to the classical case, the conjectured optimal here is $L \log^2 L$ instead of $L \log L$. This is the reasonable thing to ask for since, by \cite{Hu2009}, the noncommutative Doob maximal maps $L\log^2 L$ into $L_1[\ell_\infty]$ and that bound is known to be optimal. 
%
%
Partial answers to Problem \ref{prp:ProblemA}\ref{itm:ProblemA.2} have been obtained. For instance in \cite[Theorem A]{CondeGonPar2019} it was shown, using an argument that mimics the $\limsup$ argument from \cite{Jessen1935}, that there is bilateral almost uniform convergence for every $x \in L \log^2 L$. This result can be obtained as well from earlier results in \cite{HongSun2018}. Whether there are larger Orlicz classes than $L \log^2 L$ for which almost uniform convergence holds is an open problem, although it doesn't seem likely.

Problem \ref{prp:ProblemA}\ref{itm:ProblemA.1} can be decoupled into two halves. The first is whether the conjectured weak Orlicz type $(\Phi, \Phi)$ holds. The second is whether the Orlicz type cannot be lowered below $\Phi(s) = s \, (1 + \log_+ s)^2$. That is, if $\Psi$ is another increasing and convex function with $\Psi \in o(\Phi)$, then, $\Ee = (\EE_n \otimes \EE_m)_{n,m}$ cannot be of weak type $(\Psi, \Psi)$. Regarding the first problem, it was also shown in \cite[Theorem B3]{CondeGonPar2019} that the noncommutative strong maximal operator has weak Orlicz type $(\Theta, \Theta)$, where $\Theta(s) = L \log^{2 + \varepsilon} L$, for every $\varepsilon > 0$. Our main result in this article solves, with an argument relying on interpolation, the second part of the problem. Thus, the optimal weak Orlicz type for the noncommutative strong maximal cannot be lowered below $L \log^2 L$. To sharpen the weak Orlicz type from $L \log^{2 + \varepsilon} L$ to $L \log^2 L$ remains as a challenging open problem.





\subsection*{Our results}
As said before, this article gives a solution to one half of Problem \ref{prp:ProblemA}\ref{itm:ProblemA.1}. The key is a Marcinkiewicz-type interpolation theorem asserting that, if $\Ee = (\EE_n \otimes \EE_m)_{n,m}$ were of weak Orlicz type $(\Psi, \Psi)$, with $\Psi(s) \in o\big( s \, (1 + \log_+ s )^2\big)$, then the operator norm in $L_p$ would satisfy that 
\[
    c(p) := \big\| (\EE_n \otimes \EE_m)_{n,m}: L_p(\N \bar\otimes \M) \to L_p(\N \bar\otimes \M; \ell_\infty) \big\|
     \in o\big( (p')^4 \big)
     \;\;\; \text{ as } \; p \to 1^{+},
\]
but this would contradict the known optimal bound for the norm for the noncommutative strong maximal on $L_p$, which is, by a Fubini-type argument, the square of the optimal norm of the noncommutative Doob maximal. Thus, the optimal weak Orlicz type for the noncommutative strong maximal can't be lowered below $L \log^2 L$. 


Our main tool will be the following Marcinkiewicz-like interpolation theorem for noncommutative maximal functions.
Even though we are interesting in applications to finite von Neumann algebras, we will formulate our result in the semifinite context since ---we suspect--- it can be a useful tool with applications beyond Problem \ref{prp:ProblemA}. Here $q_{\Phi}$ and $p_{\Phi}$ denote the upper and lower Matuszewska-Orlicz indices \cite{Maligranda1985} of a Young function $\Phi$, which coincide with the usual Boyd indices \cite{Boyd1969} of the respective Orlicz spaces. We will also assume in all our statements that $q_{\Phi_1} < p_{\Phi_2}$.

\begin{ltheorem}
    \label{thm:MainTheorem}
    Let $(\N,\tau_\N)$ and $(\M,\tau_\M)$ be von Neumann algebras endowed with n.s.f. traces and let $\Phi_1,\Phi_2$ be locally finite Young functions. Given $S=(S_n)_{n\geq 1}$, a family of positive linear maps $S_n: L_1(\N) \cap L_\infty(\N) \to L_0(\M)$ such that $S$ is of maximal weak type $(\Phi_1,\Phi_1)$ and $(\Phi_2,\Phi_2)$ simultaneously, we have that 
    there exist a function $p \mapsto F(p; \Phi_1,\Phi_2) :(q_{\Phi_1},p_{\Phi_2}) \to (0,\infty)$ such that: 
    $$ 
        \big\|S: L_p(\N) \to L_p(\M;\ell_\infty) \big\|
        \, \lesssim \,
        \max \Big\{ [S]_{(\Phi_1,\Phi_1)},  \, [S]_{(\Phi_2,\Phi_2)} \Big\} \, F(p; \Phi_1,\Phi_2),       
    $$
    for every $q_{\Phi_1} < p < q_{\Phi_2}$, where the term $F(p;\Phi_1, \Phi_2)$ is given by
    \begin{equation}
        \label{eq:DefConstantep}
        \tag{${\mathrm{Ct}_p}$}
        F(p; \Phi_1,\Phi_2) :=
        \inf_{k_0} 
        \Bigg\{
            \bigg[
                \sum_{k>k_0} \left(2^{-k}  \, \big( \Phi_1(2^k) + \Phi_2(2^{k_0})\big)^{\frac1{p}}\right)^{\frac12}
                + \sum_{k\leq k_0}\left(2^{-k} \, \Phi_2(2^k)^{\frac1{p}}\right)^{\frac12}\bigg]^2
        \Bigg\}.
    \end{equation}
\end{ltheorem}
The proof of this theorem follows closely the steps of \cite{CadilhacRicard2023Inter}. In fact, in their Remark \cite[Remark 3.6]{CadilhacRicard2023Inter} it is already stated that a qualitative result like our Theorem \ref{thm:MainTheorem} would hold for rearrangement invariant spaces $E(\M,\tau)$ with $q_E$ and $p_E$ being their Boyd indices. Our addition is to turn the qualitative statement into a quantitative one by bounding the norm in terms of $p$. It is this explicit bound in terms which allows us to solve one half of Problem \ref{prp:ProblemA}\ref{itm:ProblemA.2}. 

The expression defining $F(p;\Phi_1,\Phi_2)$ can be better understood in the following manner. Let $(a_k)_{k \in \ZZ}$ and $(b_n)_{k \in \ZZ}$ be non-decreasing sequences. Their concatenation at $k_0$ can be defined as the non-decreasing sequence $a_k {\#}_{k_0} b_k$ given by $a_k$ when $k \leq k_0$ and $a_{k_0}+ b_k$ when $k > k_0$. Then, we can rewrite the expression as
\[
    F(p;\Phi_1,\Phi_2) = \inf_{k_0 \in \ZZ} \Big\| \Big( 2^{-k} \, \big( \Phi_2(2^{k}) {\#}_{k_0} \Phi_1(2^{k}) \big)^\frac1{p} \Big)_{k \in \ZZ} \Big\|_{\ell_{\frac12}(\ZZ)}.
\]
In Remark \ref{rmk: mejora exponente} we will see that the $\frac12$-quasi-norm can be improved to a ${\frac{p}{p+1}}$-quasi-norm when the operator $S$ is evaluated over projections. Nevertheless, both give the same growth order when $p \to 1^{+}$. 
Also, the maximum of $[S]_{(\Phi_1,\Phi_1)}$ and $[S]_{(\Phi_1,\Phi_1)}$ in Theorem \ref{thm:MainTheorem} above can be improved to an expression that recovers the geometric mean of both quantities when the weak types Orlicz types are just weak types $(p_1,p_1)$ and $(p_2,p_2)$ respectively, see Remark \ref{rmk: mejora constante}.
While the quantity $F(p;\Phi_1, \Phi_2)$ can be hard to compute in explicit terms for general Orlicz spaces, it can be asymptotically estimated for many reasonable Orlicz classes like $L\, \log^\alpha L$ spaces, see Corollary \ref{cor:growth order of L log L and optimality}. Before stating our result for the strong maximal, the following monotonicity theorem is needed.

\begin{llemma}
    \label{lem: Monotonia}
    Let $F(p;\Phi)$ be the constant defined by \eqref{eq:DefConstantep} when the second Orlicz type just coincides with the strong type $(\infty, \infty)$. Assume that $F(p;\Phi) \to \infty$ as $p \to q_{\Phi}^+$ and take $\Psi \in o( \Phi )$ another Young function with $q_\Phi = q_\Psi$, then $F(p;\Psi) \in o \big( F(p;\Phi) \big)$ as $p \to q_\Phi^+$.
\end{llemma}


As a corollary of the preceding monotonicity lemma and the interpolation result we get

\begin{lcorollary}
    Let $\EE_n \otimes \EE_m: \M \bar\otimes \N \to \M_n \bar\otimes \N_m$ be as before. If the noncommutative strong maximal associated to $(\EE_n \otimes \EE_m)_{n,m}$ has weak type $(\Psi, \Psi)$, then $\Psi \not\in o\big(s\,(1 + \log_+ s )^2 \big)$ 
\end{lcorollary}


\section{Preliminaries}
\numberwithin{theorem}{section}


\subsection*{Noncommutative Orlicz spaces} Let $\M \subseteq \B(\H)$ be a von Neumann algebra, we will say that it is semifinite if it admits a normal, semifinite and faithful trace $\tau: \M_+ \to [0,\infty]$, see \cite{PiXu2003, GoldsteinLabouschange2020} for the precise definitions. The pair $(\M, \tau)$ will play the part of a noncommutative or quantized measure space, with $\tau$ being analogous to the integration with respect to the measure. Given a pair $(\M, \tau)$ is possible to define the space of $\tau$-measurable operators $L_0(\M,\tau)$, as the space of unbounded operators affiliated to $\M$ satisfying that 
\[
  x \in L_0(\M,\tau)
  \;\; \equivalent \;\;
  \exists \lambda > 0, \; \tau \big( \1_{(\lambda, \infty)}(|x|)\big) < \infty,
\]
where $|x| = (x^\ast x)^\frac12$ is and $\1_{(\lambda, \infty)}$ is applied to the unbounded operator via functional calculus. The space $L_0(\M;\tau)$ is a complete metric vector space with the topology of convergence in measure, see \cite{Terp1981lp}. The noncommutative $L_p$-spaces associated to a pair $(\M,\tau)$ are simply given by 
\[
  L_p(\M, \tau) = \big\{ x \in L_0(\M,\tau) : \tau \big( |x|^p \big)< \infty \big\}
\]
while their associated norm is given by $\| x \|_p = \tau(|x|^p)^\frac1{p}$. As it is common, we will use the convention that $L_\infty(\M, \tau) = \M$ and drop the dependency on $\tau$ whenever there is no ambiguity.

A similar definition of noncommutative Orlicz spaces is available. Let $\Phi: [0,\infty) \to [0,\infty]$ be a function. We will say that it is a \emph{Young function} if
\begin{itemize}
    \item $\Phi$ is non-decreasing,
    \item $\Phi$ is convex,
    \item $\Phi(0) = 0$ and $\displaystyle \lim_{t\to \infty} \Phi(t)= \infty$.
\end{itemize}
Observe that Young functions, with the conventions that will be used throughout this text, can take the value $\infty$ over finite arguments. This may seem estrange but all the conditions, including convexity, make sense for functions with values in $[0,\infty]$ with the convention that $t \cdot \infty = \infty$ if $t \neq 0$. 

Young functions are used to define Orlicz spaces $L_\Phi(\M,\tau)$ in the following way. Let $x \in L_0(\M,\tau)$ be a $\tau$-measurable operator affiliated to $\M$. We have that $x \in L_\Phi(\M,\tau)$ if and only if, for some scalar $\lambda > 0$
\[
  \tau \left\{ \Phi \Big( \frac{|x|}{\lambda} \Big) \right\} < \infty.
\]
Moreover, $L_\Phi$ is a Banach space with norm given by
\[
  \|x\|_{\Phi} = \inf \left\{ \lambda > 0 : \; \tau \circ \Phi \Big( \frac{|x|}{\lambda} \Big) \leq 1 \; \right\}.
\]
When $(\M,\tau)$ is a measure space we have that $L_\Phi(\M,\tau)$ coincides with the classical Orlicz space $L_\Phi(\Omega)$, see \cite{Rutickii1961convex, RaoRen2002TheoryOrlicz, RaoRen2002ApplicationsOrlicz}. Notice as well that Orlicz spaces generalize $L_p$-spaces by taking $\Phi(t) = t^p$. This includes $L_\infty(\M)$ if we take the Young function $\Phi(t) = \infty \cdot \1_{[1,\infty)}(t)$.

Given two Young functions $\Phi_1$ and $\Phi_2$ we would say that $\Phi_2$ grows as fast as $\Phi_1$ if $\Phi_1(t) \lesssim \Phi_2(t)$ for $t$ larger that some quantity. Observe that, although $\lesssim$ is clearly an order relation, it is not a total order. Nevertheless, in many situations it is possible to compare a Young function $\Phi(t)$ with powers $t^\alpha$ in an optimal way, for that, we need to introduce the following indices.


\begin{definition}\label{def: definition of Orlicz indices} Let $\Phi$ be a finite positive Young function, and let $M_\Phi$ be the function:
$$
    M_\Phi(t) = \sup_{s>0} \frac{\Phi(st)}{\Phi(s)}.
$$
The (lower and upper, resp.) Matuszewska-Orlicz indices are defined as:
$$
    p_\Phi = \lim_{t\to 0} \frac{\log M_\Phi(t)}{\log t}, \qquad q_\Phi = \lim_{t\to \infty} \frac{\log M_\Phi(t)}{\log t}.
$$
\end{definition}

A few observations are in order. First, notice that the existence of the limits that define $q_\Phi$ and $p_\Phi$ is justified by the fact that $M_\Phi$ is sub-multiplicative, ie $M_\Phi(t_1  \, t_2) \leq M_\Phi(t_1) \, M_\Phi(t_2)$, together with Fekete's subadditivity lemma \cite{Fekete1923}. Next, the definition of $M_\Phi$ immediately implies that $M_\Phi(t)$ is the best constant such that 
\begin{equation}\label{eq: best doubling constant M}
    \Phi(st)\leq M_\Phi(t)\Phi(s),\qquad \text{for all } s,t>0. 
\end{equation}
The Matuszewska-Orlicz indices encode which are the best polynomial approximations of $M_\Phi(t)$ near zero and infinity, that is 
\[
    \begin{array}{rclll}
        M_\Phi(t) & = & \exp\big(q_\Phi \log t + o(\log t)\big) & \text{ as } & t \to \infty\\[2.5pt]
        M_\Phi(t) & = & \exp\big(p_\Phi \log t + o(\log t)\big) & \text{ as } & t \to 0^+
    \end{array}
\]
From these identities it is easily obtained that when $M_\Phi(t)$ goes to infinity, it grows faster than $t^{q_\Phi-\varepsilon}$ but slower than $t^{q_\Phi+\varepsilon}$ for all $\varepsilon>0$, and similarly near zero. Notice also that
\[
  \frac1{M_\Phi(t^{-1})} 
  \lesssim \Phi(t) \lesssim
  M_\Phi(t) \;\;\;\; \text{ for all } t>0,
\]
and so one can deduce also bounds on the growth and decay of $\Phi$ via $M_\Phi$. Indeed, it holds that
\begin{equation}
  \label{eq: cotas para Phi}
  \min \big\{ t^{p_\Phi - \varepsilon} , t^{q_\Phi + \varepsilon} \big\}
  \lesssim_{(\varepsilon)} 
  \Phi(t) 
  \lesssim_{(\varepsilon)} 
  \max\big\{ t^{p_\Phi - \varepsilon}, t^{q_\Phi + \varepsilon} \big\}
  \;\;\;
  \text{ for every } 
  \;
  t > 0,
\end{equation}
for every $\varepsilon>0$, see \cite{Maligranda1985}. Furthermore, $p_\Phi$ and $q_\Phi$ are the largest and smallest indices for which this holds. Further details about Matuszewska-Orlicz indices  and the function $M_\Phi$ associated to a Young function $\Phi$ can be found in \cite{Maligranda1985}.

\subsection*{Distribution functions and singular numbers} Let $x \in L_0(\M,\tau)$ be a $\tau$-measurable operator, we define the following non-negative, non-increasing functions on $s > 0$ given by
\begin{eqnarray*}
    \lambda_s(x) & = & \tau \big( \1_{(s,\infty)}(|x|) \big)\\
    \mu_s(x) & = & \inf \big\{ t > 0 : \lambda_t(x) \leq s \big\}
\end{eqnarray*}
The first one is referred to as the (noncommutative) distribution function, and the second one as the generalized singular numbers. The first can be seen as a generalization of the classical distribution function $\lambda_s(f) = \mu\big\{ |f| > s \big\}$ while the second, which is just the pseudo-inverse of the first, generalizes the decreasing re-arrangement of $f$. Both functions $\mu_s(x), \lambda_s(x): [0,\infty) \to [0,\infty]$ are non-increasing and lower semicontinuous.

The map that associates an element $x \in L_0(\M, \tau)$ with its generalized singular numbers $\mu_s(x)$ has many useful properties that generalize those of the decreasing rearrangement. For instance it holds that for every positive element $x \in L_0(\N, \tau_\N)$ 
\[
    \tau_\N(x) \, = \, \int_0^\infty \mu_s(x) \, d s.
\]
Furthermore, it holds in general that $\| x \|_p = \| \mu(x) \|_{L_p(0,\infty)}$, for every $1 \leq p \leq \infty$. The map $x \mapsto \mu(x)$ is also order preserving in the sense that, if $x \leq y$ then $\mu_s(x) \leq \mu_s(y)$ for every $s > 0$. Despite these nice properties, the map $x \mapsto \mu_s(x)$ shares certain pathologies with the classical decreasing rearrangement. In particular it fails to be sub-linear, meaning that it is not possible to bound the decreasing rearrangement of a sum by the sum of the decreasing rearrangements. This pathology can be fixed by the introduction of the following order relation. Let  $f, g: (0,\infty) \to [0,\infty]$ be two non-increasing functions, we have the majorization relation given by
\[
  f \preceq g \;\;\;\; \text{ if and only if } \;\;\;\;
  \int_0^s f(t) \, dt \leq \int_0^s g(t) \, d t, \; \forall s > 0
\]
This relation is quite useful, see \cite{CadilhacRicard2023Inter}. In particular it holds that
\begin{equation}
    \label{eq:SublinearPrec}
    \text{ if } \;\; x = \sum_{k} x_k 
    \;\;\;\; \text{ then } \;\;\;\; 
    \mu_s(x) \preceq \sum_{k} \mu_s(x_k).
\end{equation}
We will use that if $\mu_s(x) \preceq \mu_s(y)$ then $\| x \|_p \leq \| y \|_p$ for every $p \in [1,\infty]$. Both properties are easily obtained from the fact that 
\[
    \int_0^s \mu_t(x) \, dt = \| x \|_{L_1(\M) + s L_\infty(\M)}.
\] 
We will also use the notation $D_\eta$ for the dilation $(D_\eta f)(x) = f(\eta^{-1} x)$. In many situation $D_\eta \big(\mu_s(x)\big)$ would implicitly be used to mean dilation in the real variable $s$.

We will need the following lemma involving generalized singular numbers from \cite{CadilhacRicard2023Inter}. 

\begin{lemma}[{\cite[Lemma 2.1]{CadilhacRicard2023Inter}}]
    \label{lemma: Cadilhac-Ricard lemma, binary decomposition}
    Let $p \in (0,\infty)$ and $x\in L^+_p(\N)$. There exist a sequence $(r_n)_{n\in \ZZ}$ of finite projections in $\N$ so that
    \begin{enumerate}[label={\rm \textbf{(\roman*)}}, ref={\rm {(\roman*)}}]
        \item \label{itm: binary1}
        $\displaystyle{x = \sum_{n\in \mathbb{Z}} 2^{-n}r_n \in L_p(\N)}$,
        \item \label{itm: binary2}
        $\displaystyle
            \mu(x^\alpha) 
            \; \preceq \; \sum_{n \in \ZZ} 2^{-n\alpha}\mu(r_n) 
            \; \leq \; 
            \frac{1}{1-2^{-\alpha}} \mu (x^\alpha), \;
        $
        $\forall \alpha>0$.
    \end{enumerate}
\end{lemma}

The above lemma is highly constructive. Indeed, take a number $t \in \RR_+$ and consider its binary expansion. We can fix $d_n(t)$ as the characteristic function of the $t$ with a $1$ on its $-n^\text{th}$ binary digit. Identity \ref{itm: binary1} follows by spectral calculus. The first inequality in \ref{itm: binary2} follows from \eqref{eq:SublinearPrec} while the second requires an explicit computation.


\subsection*{Mixed $L_p(\ell_\infty)$-spaces}
These mixed spaces have already been covered extensively in the literature, see \cite{JunXu2003, JunPar2010}. We briefly recall that if $(x_n)_{n \geq 0}$ is a bounded sequence in $L_p(\N)$ we have that $(x_n)_n \in L_p(\N;\ell_\infty)$ if and only if there is a factorization $x_n = \alpha \, v_n \, \beta$ where $\alpha, \beta \in L_{2p}(\N)$ and $(v_n)_n \in \ell_\infty \bar\otimes \N$. The norm of the space will be given by 
\[
    \big\| (x_n)_n \big\|_{L_p(\N;\ell_\infty)} 
    :=
    \inf \Big\{ \| \alpha \|_{2p} \, \Big( \sup_n \| v_n \| \Big) \, \| \beta \|_{2p}
                : x_n = \alpha \, v_n \, \beta \Big\}.
\]
Observe that, without loss of generality, we can choose $\| \alpha \|_{2 p}^2 = \| \beta \|_{2p}^2 = \| (x_n)_n \|_{L_p[\ell_\infty]}$ and $v_n$ to be a sequence of contractions. The norm above has simpler forms when the sequence $x = (x_n)_n$ is either self-adjoint or positive. In the first case we have that it is possible to take $v_n$ self-adjoint and $\beta = \alpha^\ast$. In the case of positive sequences of operators $x_n \geq 0$, it holds that
\begin{equation}
    \label{eq:FormulaPositivos}
    \big\| (x_n)_{n} \big\|_{L_p(\N;\ell_\infty)}
    = 
    \inf \Big\{ \| \alpha \|_p : \;\; 0 \leq x_n \leq \alpha, \;\; \forall n \geq 0 \, \Big\}.
\end{equation}
We will use this formula repeatedly. Observe that, contrary to many other natural  noncommutative function spaces, the mixed $L_p[\ell_\infty]$-norm of $(x_n)_n$ is not equal to that of $(|x_n|)_n$ or $(|x_n^\ast|)_n$. In fact, the quantities associated to the $L_p[\ell_\infty]$-norm of $|x_n|$ and $|x_n^\ast|$ give rise to asymmetric mixed spaces as studied in \cite{JunPar2010, Hong2016asymmetric}.

\section{Proof of the main theorem}

We will say that two (orthogonal) projections $p$, $q$ are disjoint if $p q = 0$. In the forthcoming discussion we will need the following elementary lemma whose proof we omit. 

\begin{lemma}
    \label{lem:Corners}
    If $S$ is a positive operator on a Hilbert space $H$, and $e$ is an orthogonal projection such that $e^\perp S e^\perp = 0$, then $S =eSe$.
\end{lemma}

Similarly, the proof of our main theorem requires the following diagonal bound lemma.

\begin{lemma}[{\cite[Lemma 3.1.]{CadilhacRicard2023Inter}}]
    \label{lemma: Cadilhac-Ricard lemma, diagonal domination inequality}
    Let $(q_k)_{0\leq k\leq n}$ be a sequence of disjoint projections in $\N$. The for any sequence $(d_k)_{0\leq k \leq N}$ of positive integers and any $x\in L_0^+(\N)$ it holds that
    \begin{equation*}
        \bigg(\sum_{ k = 0}^n q_k \bigg) \, x \, \bigg( \sum_{k = 0}^n q_k \bigg)
        \leq
        \bigg(\sum_{k = 0}^n \frac{1}{ d_k}\bigg) \, \sum_{k = 0}^n d_k \, q_k \, x \, q_k
    \end{equation*}
    and the same follows for infinite sums.
\end{lemma}


We will prove Theorem \ref{thm:MainTheorem} first when $x$ is a finite projection. Then, we will decompose a general $x$ into a sum of projections following Lemma \ref{proposition: the result for projections}. Observe that, since projections are positive, we can use formula \eqref{eq:FormulaPositivos}. 

Indeed, let $r \in \N$ be a finite projection. We want to manufacture a uniform bound for $S_n(r)$ from the projection in Definition \ref{def: maximal weak Orlicz type}\ref{itm:MaxOrlicz1}. First, we will normalize the operator $S = (S_n)_{n \geq 0}$ by dividing it by the maximum of $[S]_{(\Phi_1,\Phi_1)}$ and $[S]_{(\Phi_2,\Phi_2)}$ so that, without loss of generality, we can assume that $[S]_{(\Phi_i,\Phi_i)} \leq 1$, for $i \in \{1,2\}$. Now, fix a monotone sequence of positive scalars $\lambda_k = \eta \, 2^{-k}$, where $\eta > 0$ is a parameter that we will optimize in Remark \ref{rmk: mejora constante} and that can be assumed to be one in Proposition \ref{proposition: the result for projections}. 
By definition we have that, for $i \in \{1,2\}$ and any given scalar $\lambda_k > 0$, there is a projection $e_{\lambda_k}^{i}$ satisfying that
\[
    \begin{cases}
        \displaystyle
        \big\| e_{\lambda_k}^{i} \, S_n(r) \, e_{\lambda_k}^{i} \big\|_{\infty} \leq \lambda_k \;\;\;\; \text{ for all } \; n\geq 0,\\[5pt]
        \displaystyle
        \tau \big( \1 - e_{\lambda_k}^{i} \big) \leq
        \tau \circ \Phi_i \Big(\frac{r}{\lambda_k}\Big)
    \end{cases}
\]
Now, fix an index $k_0 \in \ZZ$ such that $\Phi_2(1/\lambda_{k_0}) < \infty$. We can define a sequence of projections $(e_k)_{k \in \ZZ}$ as 
\[
  e_k \; = \;
    e^1_{\lambda_k} \, \1_{(k_0, \, \infty)}(k) +
    e^2_{\lambda_k} \, \1_{(-\infty, \, k_0]}(k).
\]
Although $(e_k)_k$ may not be decreasing, it can be used to produce a decresing sequence in a natural way as follows
\begin{equation}
    \label{eq:Def_etilde}
    \widetilde{e}_k = \bigwedge_{j \leq k}{e_j}.
\end{equation}
Since $\widetilde{e}_k$ is decreasing $q_k = \widetilde{e}_k - \widetilde{e}_{k+1}$ is a projection that should be understood as the region over which the maximal lays in $(\lambda_{k+1},\lambda_k)$. We define $z_r$ as
\begin{equation}
    \label{eq:DefMajorizer}
    z_r := \bigg(\sum_{k \in \ZZ} \frac{1}{d_k} \bigg) \, \sum_{k\in \ZZ}  d_k \, \lambda_k \, q_{k}
\end{equation}
for some $(d_k)_{k \in \ZZ}$ to be determined later. We have the following

\begin{proposition}
    \label{proposition: the result for projections}
    Let $(\N, \tau_\N)$ and $(\M, \tau_\M)$ be two semifinite von Neumann algebras and $S=(S_n)_{n \geq 0}$ be a family of operators of (maximal) weak type $(\Phi_1, \Phi_1)$ and $(\Phi_2, \Phi_2)$ simultaneously as described in Theorem \ref{thm:MainTheorem}. Given a finite projection $r \in \N$,
    $z_r$, defined as in \eqref{eq:DefMajorizer}, satisfies that
    \begin{enumerate}[label={\rm \textbf{(\roman*)}}, ref={\rm {(\roman*)}}]
        \item \label{itm: Projections0} The projections in \eqref{eq:Def_etilde} satisfy that $\big\| \widetilde{e}_k \, S_n(r) \, \widetilde{e}_k \big\|_\infty \leq \lambda_k$ and that
        \begin{equation}
            \label{eq: Cota para las e_k tilde}
            \tau(\1 - \widetilde{e}_k) 
            \, \leq \;
            2
            \, 
            \begin{cases}
                \displaystyle{
                (\tau \circ \Phi_2) \Big( \frac{r}{\lambda_k }\Big)} & \text{ when } k \leq k_0\\[12pt]
                \displaystyle{
                (\tau \circ \Phi_2) \Big( \frac{r}{\lambda_{k_0}} \Big) + (\tau \circ \Phi_1) \Big( \frac{r}{\lambda_k} \Big)} & \text{ when } k > k_0
            \end{cases}
        \end{equation}
        \item \label{itm: Projections1} $S_n(r) \leq z_r$, for every $n \geq 0$.
        \item \label{itm: Projections2} For some $k_0 \in \ZZ$ as above and a choice of $d_k$ independent of $r$ it holds that
        \begin{equation}
            \label{eq:ProjectionsBound}
            \begin{split}
            \mu(z_r)
            \preceq 
            \bigg( \sum_{k \in \ZZ} \frac1{d_k} \bigg) 
            \Bigg[
            \sum_{k \geq k_0} d_k \, \lambda_k \, & D_{2 \, \Phi_2(\lambda_{k_0}^{-1}) + 2\, \Phi_1(\lambda_{k+1}^{-1})} (\mu(r))  \\
            & + 
            \sum_{k < k_0} d_k \, \lambda_k \, D_{2\, \Phi_2(\lambda_{k+1}^{-1})} (\mu(r))
            \Bigg],
            \end{split}
        \end{equation}
        
        
        \item \label{itm: Projections3} For every $q_{\Phi_1} < p < p_{\Phi_2}$, it is possible to optimize $(d_k)_k$ to obtain that
        $$
            \|z_r\|_p 
            \lesssim 
            \max \big\{ [S]_{(\Phi_1,\Phi_1)},  \, [S]_{(\Phi_2,\Phi_2)} \big\} \, F(p;\Phi_1,\Phi_2) \, \| r \|_p,
        $$
        where $F(p,\Phi_1,\Phi_2)$ is the factor defined in \eqref{eq:DefConstantep}.
    \end{enumerate}
\end{proposition}



\begin{proof}

For \ref{itm: Projections0}, start by noticing that, for any finite projection $r \in \N$ and Young function $\Phi$, it holds by spectral calculus that
\begin{equation*}
    \tau \circ \Phi \Big(\frac{r}{\lambda} \Big)
    = \tau(r) \, \Phi \Big( \frac1{\lambda} \Big).
\end{equation*}
First, let $k \leq k_0$. The fact that $\big\| \widetilde{e}_k \, S_n(r) \, \widetilde{e}_k \big\|_\infty \leq \lambda_k$ is trivially verified in any case. For the second, we have that
\begin{eqnarray*}
    \tau\big(\1-\widetilde e_k\big) \; = \; \sum_{j \leq k} \tau\big(\1- e_j\big)
      & \leq & \sum_{j \leq k} \Phi_2\Big( \frac1{\lambda_j} \Big) \, \tau(r)  \\
      & \leq & \bigg( \sum_{j \leq k} M_{\Phi_2}\Big( \frac{\lambda_k}{\lambda_j} \Big)  \bigg) \, {\Phi_2}\Big( \frac1{\lambda_k} \Big)\, \tau(r),
\end{eqnarray*}
where we have multiplied and divided by $\lambda_k$ inside the argument of $\Phi_2$ and used identity \eqref{eq: best doubling constant M}. By convexity, we have that $M_{\Phi}(s) \leq s$ for every $0 \leq s \leq 1$, thus, the constant is bounded by
\[
  \sum_{j \leq k} \frac{\lambda_k}{\lambda_j}
  = \sum_{j \leq 0} 2^{j} = 2,
\]
which is independent of $k$. Observe that, in the case of a non-finite $\Phi_2$, we still have that $\Phi_2(s \, t ) \leq s \, \Phi_2(t)$ for $0 \leq s \leq 1$ and the computation above is still valid.
Now we need to tackle the case of $k_0 < k$. In that case, we have that
\begin{eqnarray*}
    \tau\big(\1-\widetilde e_k\big) & = & \sum_{j \leq k} \tau\big(\1- e_j\big)\\
      & = & \sum_{j \leq k_0} \Phi_2\Big( \frac1{\lambda_j} \Big) \, \tau(r) \; + \; \sum_{k_0 < j \leq k} \Phi_1\Big( \frac1{\lambda_j} \Big) \, \tau(r)\\
      & \leq & \sum_{j \leq k_0} M_{\Phi_2} \Big( \frac{\lambda_{k_0}}{\lambda_j} \Big) \, \Phi_2\Big( \frac1{\lambda_{k_0}} \Big) \, \tau(r) \; + \; \sum_{k_0 < j \leq k} M_{\Phi_1} \Big( \frac{\lambda_{k}}{\lambda_{j}}\Big)\Phi_1\Big( \frac1{\lambda_{k}} \Big) \, \tau(r)\\
      & \leq & 2 \, \bigg[ \Phi_2\Big( \frac1{\lambda_{k_0}} \Big) + \Phi_1\Big( \frac1{\lambda_{k}} \Big)\bigg] \, \tau(r).
\end{eqnarray*}
%
For \ref{itm: Projections1}, first notice that the limit projection $\widetilde{e}_{\infty} = \inf_{k} \widetilde{e}_k$ verifies that $\widetilde{e}_{\infty} \, S_n(r) \, \widetilde{e}_{\infty} = 0$ for all $n$. Since $S_n(r)$ is positive, by Lemma \ref{lem:Corners} it holds that $\widetilde{e}^\perp_\infty \, S_n(r) \, \widetilde{e}^\perp_\infty = S_n(r)$. On the other hand, since $\displaystyle{\lim_{t\to 0} \Phi_2(t)= 0}$, we have that $\widetilde{e}_k \to \1$ weakly as $k \to -\infty$. Thus, the projections
$$
    g_N = \sum_{-N\leq k \leq N} q_k
$$
satisfy that $g_N \to \1-\widetilde{e}_\infty$. Therefore $g_N \, S_n(r) \, g_N \to S_n(r)$. Applying Lemma \ref{lemma: Cadilhac-Ricard lemma, diagonal domination inequality} as well as \ref{itm: Projections0} above gives
\begin{equation*}
\label{eq: definition of z_r}
\begin{split}
    g_N \, S_n(r) \, g_N 
    & \leq 
    \bigg( \sum_{|k| \leq N} \frac{1}{d_k}\bigg) \sum_{|k|\leq N} d_k \, q_k \, S_n(r) \, q_k \\
    & \leq
    \bigg( \sum_{|k| \leq N} \frac{1}{d_k}\bigg) \sum_{|k| \leq N}  d_k \, \lambda_k \, q_k 
\end{split}
\end{equation*}
and taking limits as $N \to \infty$ yields the bound \ref{itm: Projections1}. 

It holds that $q_k \leq \1 - \widetilde{e}_{k+1}$, thus point \ref{itm: Projections2} follows from an application of the $L_p$-convergence of the sum defining $z_r$ in addition to the bound:
\begin{eqnarray*}
    \mu(q_{k-1}) = \1_{[0,\tau(q_{k-1})]}
        & \leq & \1_{[0,\tau(\1 - \tilde{e}_{k})]}\\
        & \leq & D_{2 \, \Phi_2(\lambda_{k_0}^{-1}) + 2 \, \Phi_1(\lambda_k^{-1})} (\mu(r)) \, \1_{(k_0, \infty)}(k)+ D_{2 \, \Phi_2(\lambda_k^{-1})} (\mu(r)) \, \1_{(-\infty,k_0]}(k).
\end{eqnarray*}
Using that the relation $\prec$ is sub-additive, ie \eqref{eq:SublinearPrec}, we get the desired identity.

For \ref{itm: Projections3}, we use the triangular inequality together with the fact that $\tau(q_k) \leq \tau(\1 - \tilde{e}_{k+1})$ and the bounds in \eqref{eq: Cota para las e_k tilde} to obtain that
\begin{eqnarray}
    \|z_r\|_p & = & \bigg\| c_d \, \sum_{k \in \ZZ} d_k \, \lambda_k \, q_k \bigg\|_p \\
        & \leq & c_d \, \sum_{k \in \ZZ} d_k \, \lambda_k \, \tau(\1 - \widetilde{e}_{k+1})^\frac1{p} \nonumber
\end{eqnarray}
Observe that, making the change of variable $k \mapsto k-1$ and renaming the $d_{k-1}$ by $d_k$, we have that 
\begin{eqnarray}
    \|z_r\|_p
        & \leq & 2 \, c_d \, \sum_{k \in \ZZ} d_k \, \lambda_k \, \tau(\1 - \widetilde{e}_{k})^\frac1{p}\\
        & \leq & 4 \, \bigg( \sum_{k \in \ZZ} \frac1{d_k} \bigg) \, \bigg[ \sum_{k > k_0} d_k \, \lambda_k \, \bigg( \Phi_2\Big( \frac1{\lambda_{k}} \Big) + \Phi_1\Big( \frac1{\lambda_{k}} \Big) \bigg)^\frac1{p} + \sum_{k \leq k_0} d_k \, \lambda_k \, \Phi_2\Big( \frac1{\lambda_{k}} \Big)^\frac1{p} \bigg] \nonumber
\end{eqnarray}
Now we need to choose a sequence $(d_k)_k$ that is close to a minimizer for the functional above. Since one of the multiplicative terms increases with $(d_k)_k$ and the other decreases, we can try making the two terms similar, this yield the following choice of coefficients
\[
    d_k 
    =
    \lambda_k^{-\frac12} \, \bigg( \Phi_2\Big(\frac{1}{\lambda_{k}}\Big) + \Phi_1\Big(\frac{1}{\lambda_{k}}\Big) \bigg)^{-\frac1{2 p}} \, \1_{(k_0, \infty)}(k) 
    + 
    \lambda_k^{-\frac12} \,  \Phi_2\Big(\frac{1}{\lambda_{k}}\Big)^{-\frac1{2p}} \, \1_{(-\infty, k_0]}(k).
\]
Substituting in the equation gives that
\begin{eqnarray}
    \|z_r\|_p 
        & \leq & 4 \, 
          \Bigg[ 
          \,
          \sum_{k>k_0} \lambda_k^\frac12 \, \bigg( \Phi_1\Big(\frac{1}{\lambda_{k}}\Big) + \Phi_2\Big(\frac{1}{\lambda_{k_0}}\Big) \bigg)^{\frac1{2p}} 
          +
          \sum_{k\leq k_0} \lambda_k^\frac12 \, \Phi_2\Big(\frac{1}{\lambda_{k}}\Big)^{\frac1{2p}}
          \,
          \Bigg]^2 \|r\|_p. \label{eq: bound of norm p of z_r}
\end{eqnarray}
We have chosen $k_0$ such that $\Phi_2(\lambda_{k_0}^{-1}) < \infty$ and therefore, we can take infimum on the right hand side over $k_0$ satisfying that condition. If $\Phi_2$ is finite everywhere, we already have the required statement. Otherwise, the $k_0$'s for which $\Phi_2(\lambda_{k_0}^{-1}) = \infty$ give a right hand side that is infinite, and thus the infimum is unaffected. Since we have re-scaled our operator $S$, a simple scaling argument gives the factor $\max\big\{ [S]_{(\Phi_1,\Phi_1)}, [S]_{(\Phi_2,\Phi_2)}\big\}$. It only rest to see that $F(p;\Phi_1,\Phi_2)$ is finite when $q_{\Phi_1} < p < p_{\Phi_2}$. For simplicity, we can both multiply and divide the term $\Phi_2(\lambda_{k_0}^{-1})$ by $\Phi_1(\lambda_{k}^{-1})$, then using that $\Phi_1(\lambda_{k_0}^{-1}) \leq \Phi_1(\lambda_{k}^{-1})$, we have that
\[
    F(p;\Phi_1,\Phi_2) 
    \; \leq \;
    \bigg( 1 + \, \frac{\Phi_2(\lambda_{k_0}^{-1})}{\Phi_1(\lambda_{k_0}^{-1})} \bigg)^{\frac1{p}}
    \Bigg[ 
          \,
          \sum_{k>k_0} \lambda_k^\frac12 \, \Phi_1\Big(\frac{1}{\lambda_{k}}\Big)^{\frac1{2p}} 
          +
          \sum_{k\leq k_0} \lambda_k^\frac12 \, \Phi_2\Big(\frac{1}{\lambda_{k}}\Big)^{\frac1{2p}}
          \,
    \Bigg]^2
\]
Since $q_{\Phi_1} < p < p_{\Phi_2}$, there is an $\varepsilon > 0$ such that $q_{\Phi_1} + \varepsilon < p < p_{\Phi_2} - \varepsilon$. An application of \eqref{eq: cotas para Phi} shows that $\Phi_1(\lambda_k^{-1})\leq 2^{k(q_{\Phi_1}+\varepsilon)}$ for $k$ positive and big enough, as well as $\Phi_2(\lambda_k^{-1})\leq 2^{k(p_{\Phi_2}-\varepsilon)}$ when $k$ is negative and big. All this together gives
\begin{equation*}
    F(p;\Phi_1, \Phi_2) 
    \lesssim_{(\varepsilon, \Phi_1, \Phi_2, k_0)} 
    \Bigg[ 
    \,
    \sum_{k>k_0} 2^{\frac{k}{2} \big( \frac{q_{\Phi_1}+\varepsilon}{p} - 1\big)}
    +
    \sum_{k\leq k_0} 2^{\frac{k}{2} \big( \frac{p_{\Phi_2}-\varepsilon}{p} - 1\big)} 
    \,
    \Bigg]^2
    < \infty  
\end{equation*}
and that concludes \ref{itm: Projections3}.
\end{proof}
\begin{proof}[Proof of Theorem \ref{thm:MainTheorem}]
We only have to prove the result for $x\geq 0$ in $L_p$. 
Using Lemma \ref{lemma: Cadilhac-Ricard lemma, binary decomposition} we find a decomposition $x = \sum 2^{-m}r_m$, with $r_m$ a finite spectral projection of $x$. Pick $z_m$ such that $S_n(r_m)\leq z_m$ as in Proposition \ref{proposition: the result for projections}. Define $z = \sum 2^{-m}z_m$, so $S_n(x)\leq z$. To see that $z$ is well defined, we just have to show that the sum $\sum 2^{-m}z_m$ converges absolutely in $L_p$-norm. Actually, we'll have to prove the convergence of a coarser sum a few steps ahead, so we omit this step.

We control first the decreasing rearrangement of $z$ using \eqref{eq:SublinearPrec}, Proposition
\ref{proposition: the result for projections}\ref{itm: Projections2} and Lemma \ref{lemma: Cadilhac-Ricard lemma, binary decomposition}\ref{itm: binary2} for $\alpha=1$:
\begin{eqnarray*}
    \mu(z) & \preceq & \sum_{m\in \mathbb{Z}} 2^{-m} \mu(z_m)\\
        & \preceq & c_d \, \bigg(\sum_{k>k_0}\sum_{m\in \mathbb{Z}}2^{-m}d_k\lambda_k D_{2 \, \Phi_2(\lambda_{k_0}^{-1}) + 2 \, \Phi_1(\lambda_k^{-1})} (\mu(r_m))+ \sum_{k\leq k_0}\sum_{m\in \mathbb{Z}} 2^{-m} \, d_k \, \lambda_k \, D_{2 \, \Phi_2(\lambda_k^{-1})} (\mu(r_m))\bigg)\\
        & = & c_d \, \bigg(\sum_{k>k_0} \, d_k \, \lambda_k \, D_{2 \, \Phi_2(\lambda_{k_0}^{-1}) + 2 \, \Phi_1(\lambda_k^{-1})} \Big( \sum_m 2^{-m}\mu(r_m)\Big) + \sum_{k\leq k_0} d_k \, \lambda_k \, D_{2 \, \Phi_2(\lambda_k^{-1})} \Big(\sum_m 2^{-m}\mu(r_m)\Big)\bigg) \\
        & \leq  & 2 \,  c_d \, \bigg(\sum_{k>k_0} d_k \, \lambda_k \, D_{2 \, \Phi_2(\lambda_{k_0}^{-1}) + 2 \, \Phi_1(\lambda_k^{-1})} (\mu(x)) + \sum_{k\leq k_0} d_k \, \lambda_k \, D_{2 \, \Phi_2(\lambda_k^{-1})} (\mu(x))\bigg)
\end{eqnarray*}
Now we compare the $L_p$-norms of $z$ and $x$:
\begin{eqnarray*}
    \|z\|_p & = & \| \mu(z) \|_p\\
        & \leq & 2 \, c_d \,  \bigg(\sum_{k>k_0} d_k \, \lambda_k \, \big\|D_{2 \, \Phi_2(\lambda_{k_0}^{-1}) + 2 \, \Phi_1(\lambda_k^{-1})} (\mu(x))\big\|_p + \sum_{k\leq k_0} d_k \, \lambda_k \, \big\|D_{2 \, \Phi_2(\lambda_k^{-1})} (\mu(x))\big\|_p\bigg)\\
        & \leq    & 4\, c_d \, \bigg(\sum_{k>k_0} d_k \, \lambda_k \, \bigg( \Phi_1\Big(\frac1{\lambda_k}\Big) + \Phi_2\Big(\frac1{\lambda_{k_0}}\Big) \bigg)^\frac1{p} + \sum_{k\leq k_0} d_k \, \lambda_k \, \Phi_2\Big(\frac1{\lambda_k}\Big)^\frac1{p} \bigg) \, \|x\|_p\\
        & \lesssim & F(p; \Phi_1,\Phi_2) \, \|x\|_p,
\end{eqnarray*}
where $F(p; \Phi_1,\Phi_2)$ is the same function that appears in equation \eqref{eq: bound of norm p of z_r}.
\end{proof}

There are several things to clarify in the preceding proofs. Let us start saying that, although in Proposition \ref{proposition: the result for projections} we use $\lambda_k = 2^{-k}$ all the computations work verbatim with $\lambda_k = \eta \, 2^{-k}$ for some $\eta > 0$, the advantage of this extra degree of freedom is that now we can optimize the constant in terms of $[S]_{(\Phi_1, \Phi_1)}$ and $[S]_{(\Phi_2, \Phi_2)}$.

\begin{remark} \normalfont
    \label{rmk: mejora constante}
    Both in the classical Marcinkiewicz theorem and in the noncommutative one, when the Young functions are given by $\Phi_i(t) = t^{p_i}$, we have that the dependency of the operator $L_p$-norm of $S = (S_n)_{n \geq 0}$ on the weak type quasi norms $[S]_{(\Phi_i, \Phi_i)}$ is not a supremum but a geometric mean, ie:
    \[
      \big\| S: L_p(\N) \to L_p(\M;\ell_\infty) \big\|
      \; \lesssim_{(p, \Phi_1, \Phi_2)} \;
      [S]_{(p_1,p_1)}^\theta \, [S]_{(p_2, p_2)}^{1 - \theta}
      \;\; 
      \text{ for }
      \;\;
      \frac1{p} = \frac{\theta}{p_1} + \frac{1-\theta}{p_2}.
    \]  
    A similar dependency would be desirable when interpolating between two weak Orlicz types, changing of course the exponents $p_1$ and $p_2$ by the adequate Boyd indices. 
    For the sake of brevity we will denote the weak Orlicz type quasi-norms of $S$ with respect to $\Phi_1$ and $\Phi_2$ by $C_1$ and $C_2$ respectively. 
    Assume that $\Phi_2$ is finite everywhere. 
    Choosing $\lambda_k = \eta \, 2^{-k}$ and repeating the computations on Proposition \ref{proposition: the result for projections} gives that
    \begin{equation*}
        \tau( \1 - \widetilde{e}_k)
        \; \leq \;
        2\,
        \begin{cases}
            \displaystyle
            \Phi_2 \Big( C_2 \frac{2^{k}}{ \eta }\Big) \, \tau(r) & \text{ if } k \leq k_0,\\[10pt]
            \displaystyle
            \bigg[ \Phi_2 \Big( C_2 \frac{2^{k_0}}{\eta} \Big) + \Phi_1 \Big( C_1 \frac{2^k}{\eta} \Big) \bigg] \, \tau(r) & \text{ if } k > k_0.
        \end{cases}
    \end{equation*}
    Repeating the computations in the bound of $z_r$ yields that
    \begin{eqnarray*}
        \| z_r \|_p 
            & \lesssim & c_d \, \bigg[ \sum_{k \leq k_0} d_k \, \eta \, 2^{-k} \, \Phi_2\Big( C_2 \frac{2^{k}}{\eta} \Big)^\frac1{p} + \sum_{k > k_0} d_k \, \eta \, 2^{-k} \, \bigg( \Phi_2\Big( C_2 \frac{2^{k_0}}{\eta} \Big) + \Phi_1\Big(C_1 \frac{2^{k}}{\eta} \Big) \bigg)^\frac1{p} \bigg] \, \| r \|_p\\
            & \lesssim & c_d \, \bigg[ \eta \, M_{\Phi_1}\Big( \frac{C_1}{\eta} \Big)^\frac1{p}+\eta \, M_{\Phi_2}\Big( \frac{C_2}{\eta} \Big)^\frac1{p} \bigg] \\
            & & \;\;\;\;\;\;\; \bigg[ \sum_{k \leq k_0} d_k \, 2^{-k} \, \Phi_2( 2^{k} )^\frac1{p} + \sum_{k > k_0} d_k \, 2^{-k} \, \big( \Phi_2(2^{k_0}) + \Phi_1( 2^{k} ) \big)^\frac1{p} \bigg] \, \| r \|_p
    \end{eqnarray*}
    Minimizing the expression above in $d_k$ and following closely the computations on 
    the proof of Theorem \ref{thm:MainTheorem} yields that, for every $q_{\Phi_1} < p < p_{\Phi_2}$
    \[
        \big\| S : L_p(\N) \to L_p(\M;\ell_\infty) \big\|
        \lesssim_{(\Phi_1, \Phi_2)} \inf_{\eta > 0} \bigg[ \eta \, M_{\Phi_1}\Big( \frac{C_1}{\eta} \Big)^\frac1{p}+\eta \, M_{\Phi_2}\Big( \frac{C_2}{\eta} \Big)^\frac1{p} \bigg] \, F(p;\Phi_1, \Phi_2).
    \]
    When $M_{\Phi_1}(t) = t^{p_1}$ and $M_{\Phi_2}(t) = t^{p_2}$, the infimum above can be computed in terms of $\eta > 0$ and it yields precisely the geometric mean of $C_1$ and $C_2$ appearing in the classical case. This is useful in many cases. Similar computations can be repeated in the case of an operator of weak type $(\Phi_1, \Phi_1)$ and strong type $(\infty, \infty)$, in that case, we have that, taking $k_0 = 0$ for simplicity, the inequality \eqref{eq: Cota para las e_k tilde} yields $0$ in the first line as long as $\eta > C_2$. Therefore, the dependency on $C_1$ and $C_2$ becomes
    \[
      \big\| S : L_p(\N) \to L_p(\M;\ell_\infty) \big\|
        \lesssim_{(\Phi_1)} \inf_{\eta > C_2} \bigg\{ \eta \, M_{\Phi_1}\Big( \frac{C_1}{\eta} \Big)^\frac1{p} \bigg\} \, \bigg[\sum_{k\geq 0} 2^{-\frac{k}{2}}  \, \Phi_1(2^k)^{\frac1{2 p}} \bigg]^2
    \]
    and again the infimum on $\eta$ yields the desired geometric mean when $M_{\Phi_1}(t) = t^{p_1}$.
  \end{remark}

\begin{remark} \normalfont
    \label{rmk: mejora exponente}
    An important observation is that the computations in Proposition \ref{proposition: the result for projections} are not optimal. Indeed, it is possible to improve the asymptotic behaviour of the constant in terms of $p$ when we evaluate $S$ over projections. Let $z_r$ be the element defined by \eqref{eq:DefMajorizer}. Since it is given by scalar combinations of orthogonal projections we can actually estimate its $L_p$-norm explicitly
    \begin{eqnarray}
        \|z_r\|_p 
        & = & c_d \, \bigg[\sum_{k \in \ZZ} d_k^p \, \lambda_k^p \, \tau(\1 - \widetilde{e}_{k+1}) \bigg]^\frac1{p} \nonumber \\
        & \lesssim & \bigg( \sum_{k \in \ZZ} \frac1{d_k} \bigg) \, \bigg[ \sum_{k > k_0} d_k^p \, \lambda_k^p \, \bigg( \Phi_2\Big( \frac1{\lambda_{k_0}} \Big) + \Phi_1\Big( \frac1{\lambda_{k}} \Big) \bigg) + \sum_{k \leq k_0} d_k^p \, \lambda_k^p \, \Phi_2\Big( \frac1{\lambda_{k}} \Big) \bigg]^\frac1{p} \, \| r \|_p \label{eq:Bound for z_r in ell_p}
    \end{eqnarray}
    Then optimizing $d_k$ with the following choice of $d_k$ 
    \begin{equation}
        \label{eq: nuevos d_k}
        d_k 
        \; = \;
        \lambda_k^{-\frac{p}{p+1}} \bigg( \Phi_1\Big( \frac1{\lambda_k} \Big) + \Phi_2 \Big( \frac1{\lambda_{k_0}} \Big) \bigg)^{\frac{-1}{p+1}} \, \1_{(k_0,\infty)}(k) + \lambda_k^{-\frac{p}{p+1}} \, \Phi_2 \Big( \frac1{\lambda_{k}} \Big)^{\frac{-1}{p+1}} \, \1_{(-\infty,k_0]}(k)
    \end{equation}
    gives that
    \begin{equation}
      \label{eq:improved bound}
      \|z_r\|_p
      \leq
          \Bigg[ 
          \,
          \sum_{k>k_0} \lambda_k^\frac{p}{p+1} \, \bigg( \Phi_2\Big( \frac1{\lambda_{k_0}} \Big) + \Phi_1\Big( \frac1{\lambda_{k}} \Big) \bigg)^{\frac1{p+1}} 
          +
          \sum_{k\leq k_0} \lambda_k^\frac{p}{p+1} \, \Phi_2\Big(\frac{1}{\lambda_{k}}\Big)^{\frac1{p+1}}
          \,
          \Bigg]^\frac{p+1}{p} \, \| r \|_p
    \end{equation}
    This allows us to refine the statement of Proposition \ref{proposition: the result for projections}\ref{itm: Projections3} by changing $F(p;\Phi_1,\Phi_2)$ by a function $G(p;\Phi_1,\Phi_2)$ given by
    \begin{eqnarray}
        G(p;\Phi_1, \Phi_2) 
        & = &
        \inf_{k_0} 
        \Bigg\{
            \bigg[\sum_{k>k_0}\left(2^{-k}  \, \big( \Phi_2(2^{k_0}) +\Phi_1(2^k) \big)^{\frac1{p}}\right)^{\frac{p}{p+1}}
            +\sum_{k\leq k_0}\left(2^{-k} \, \Phi_2(2^k)^{\frac1{p}}\right)^{\frac{p}{p+1}}\bigg]^\frac{p+1}{p}
        \Bigg\}. \nonumber\\
         & = &
        \inf_{k_0} 
        \bigg\{
            \Big\| \Big( 2^{-k} \big( \Phi_2(2^{k}) {\#}_{k_0} \Phi_1(2^{k}) \big)^\frac1{p} \Big)_{k \in \ZZ} \Big\|_{{\frac{p}{p+1}}}
        \bigg\},
    \end{eqnarray}
    where $\Phi_1(2^k){\#}_{k_0}\Phi_2(2^k)$ is the concatenation of the sequences $(\Phi_1(2^k))_{k\leq k_0}$ and $(\Phi_1(2^k)+\Phi_2(2^k))_{k>k_0}$ as defined in the Introduction. Observe that, when $\Phi_1(t) = t^{p_1}$ and $\Phi_2(t) = t^{p_2}$, this expression becomes 
    \[
      G(p;\Phi_1, \Phi_2) \in O \bigg( \Big( \frac{1}{p - p_1} + \frac{1}{p_2 - p}\Big)^\frac{p+1}{p} \bigg).
    \]
    It is unclear to us whether this behaviour holds for the whole norm of $S: L_p(\M) \to L_p(\N;\ell_\infty)$ and not just its restriction to projections, see \cite[Remark 3.10]{CadilhacRicard2023Inter} for more on the optimal growth of the constant.
\end{remark}

\section{Applications}
In this section we prove that the weak Orlicz type of the maximal operator associated to $(\EE_n \otimes \EE_m)_{n,m}$ can not be improved to $\Phi_\alpha(s) = s \, (1 + \log_+ s)^\alpha$ for $\alpha<2$, nor to any other Orlicz space $L_\Psi$ with $\Psi\in o(\Phi_2)$.
This result is contained in Corollary \ref{cor:growth order of L log L and optimality} below, which we include since it gives a quantitative version of Theorem \ref{thm:MainTheorem} just in terms of the logarithmic exponent. 
In the proof we crucially use the Lemma \ref{lem: Monotonia} from the Introduction, which we now prove.

\begin{proof}[Proof of Lemma \ref{lem: Monotonia}]
    Recall that the space $L_\infty$ coincides with the Orlicz space $L_{\chi_\infty}$ generated by $\chi_\infty = \infty \cdot 1_{(1,\infty)}$.
    The function $F(p;\Phi, \chi_\infty)$ from Theorem \ref{thm:MainTheorem} has the form
    \begin{eqnarray*}
        F(p;\Phi, \chi_\infty) 
            & = & \inf \bigg\{ \bigg[\sum_{k > k_0} 2^{-\frac{k}{2}}  \, \Phi(2^k)^{\frac1{2 p}} \bigg]^2 : \;\; k_0 \; \text{ s.t. } \chi_\infty(2^{k_0}) < \infty \, \bigg\},\\
            & = & \bigg[\sum_{k\geq 0} 2^{-\frac{k}{2}}  \, \Phi(2^k)^{\frac1{2 p}} \bigg]^2,
    \end{eqnarray*}
    and the same formula holds for $F(p;\Psi,\chi_\infty)$.
    
    We need to prove that 
    \[
      \lim_{p \to q_\Phi^+} \; \frac{F(p; \Psi, \chi_\infty)}{F(p; \Phi, \chi_\infty)} = 0.
    \]
    First let's define for every $r\in (0,1)$ the sets
    \begin{equation*}
    \begin{split}
        A_n(r) &= \{k\geq 0\colon \psi(2^k)\leq r^n\varphi(2^k)\} \quad \text{ for } n\geq 1,\\
        E_n(r) &=
        \begin{cases}
            A_1(r)^c & \text{if } n = 0, \\
            A_n(r) \smallsetminus A_{n+1}(r) & \text{if } n\geq 1.
        \end{cases}
    \end{split}
    \end{equation*}

    The family of sets $(E_n(r))_{n\geq 0}$ is a numerable partition of the positive integers into finite subsets. It also holds that
    \begin{equation*}
    \begin{split}
        F(p; \Psi, \chi_\infty)^{1/2}
            &\leq
            \sum_{k\in E_0(r)} 2^{-k/2}\Psi(2^k)^{1/2p}
            +
            \sum_{n\geq 1} r^{n/2p}\sum_{k\in E_n(r)} 2^{-k/2}\Phi(2^k)^{1/2p}\\
            &\leq
            \sum_{k\in E_0(r)} 2^{-k/2}\Psi(2^k)^{1/2p}
            +
            \left( \frac{r^{1/2p}}{1-r^{1/2p}}\right)F(p; \Phi, \chi_\infty)^{1/2}.
    \end{split}
    \end{equation*}
    Since the left term is just a finite sum, it is bounded for all possible $p$. This proves that for all $r\in (0,1)$,
    \begin{equation*}
        \limsup_{p\to q_\Phi^+} \frac{F(p; \Psi, \chi_\infty)}{F(p; \Phi, \chi_\infty)} \leq \left(\frac{r^{1/2q_\Phi}}{1-r^{1/2q_\Phi}}\right)^2,
    \end{equation*}
    which means that the limit must be $0$.
\end{proof}

\begin{corollary}\label{cor:growth order of L log L and optimality}
    Let $S=(S_n)_{n \geq 0}$ be family of positivity-preserving operators like those of Theorem \ref{thm:MainTheorem}. Assume they are of maximal weak type $(\Phi_\alpha, \Phi_\alpha)$ and strong type $(\infty, \infty)$, where $\Phi_\alpha(s) = s \, (1 + \log_+ s)^\alpha$ and $\alpha > 0$. Then
    \begin{enumerate}[label={\rm \textbf{(\roman*)}}]
        \item $S=(S_n)_{n\geq 0}$ is of (maximal) strong type $(p,p)$, and the function 
        \[
            c(p) := \big\| S: L_p(\N) \to L_p(\M;\ell_\infty) \big\| 
            \in 
            O\big(\left(p'\right)^{2+\alpha}\big)
            \;\;\;\; \text{ as } \;\; p \to 1^+
        \]
        \item As a consequence, the strong maximal operator on Problem \ref{prp:ProblemA} is not of weak Orlicz type $(\Phi_\alpha,\Phi_\alpha)$ for any $\alpha < 2$. 
    \end{enumerate}
    
\end{corollary}
\begin{proof}
So from Theorem \ref{thm:MainTheorem} we know that the function
\begin{equation*}
    F(p;\Phi_\alpha, \chi_\infty) = 
            \bigg(\sum_{k\geq 0}2^{-k(\frac{1}{2} - \frac{1}{2p})} (1 + k)^{\frac{\alpha}{2p}}\bigg)^2
\end{equation*}
verifies $\big\| S: L_p(\N) \to L_p(\M;\ell_\infty) \big\| \lesssim F(p;\Phi_\alpha, \chi_\infty)$ for all $p\in (1,\infty)$, with a constant just depending on the Orlicz quasi-norms.

The first part of the statement follows from the fact that $F(p;\Phi_\alpha, \chi_\infty)\in O(\left(p'\right)^{2+\alpha})$. To prove this, define $r(p)= 2^{\frac{1}{2} - \frac{1}{2p}}$
and $\alpha(p) = \frac{\alpha}{2p}$. The function $F(p;\Phi_\alpha, \chi_\infty)$ has the same order as
$$
    \bigg(\sum_{k\geq 0} r(p)^{-k} k^{\alpha(p)}\bigg)^2 \in O\Big(\frac{1}{(1-r(p))^{2(\alpha(p)+1)}}\Big).
$$
And a routine check shows that this term is $O\big({(1-\frac1{p})^{2+\alpha}}\big)$, proving the claim.

The second part of the statement is just an application of the fact that function assigning each $p$ to the $p$-norm of the aforesaid maximal operator is known not to be in $o((p')^4)$, see \cite{Hu2009}.
\end{proof}

\subsubsection*{Acknowledgment} In a unreleased draft of this article the authors claimed, wrongly, that the bound in \eqref{eq:improved bound} in Remark \ref{rmk: mejora exponente} could be extended beyond projections to every element in $L_p$. The authors are thankful to Éric Ricard and Léonard Cadilhac for a helpful discussion that led to the clarification of that point.





\bibliographystyle{plain}
\bibliography{bibliography}    

\end{document}